\documentclass[10pt]{article}
\usepackage{amsmath, amsthm,amssymb}
\usepackage{bm}
\usepackage{bbm}
\usepackage[hidelinks]{hyperref} 
\usepackage{makeidx}
\usepackage{enumerate}
\usepackage{stmaryrd}
\title{A remark on the ultrapower algebra of the hyperfinite factor}
\author{Ionut Chifan and Sayan Das}

\newtheorem{thm}{Theorem}[section]

\newtheorem{prop}[thm]{Proposition}

\newtheorem{lem}[thm]{Lemma}

\newtheorem*{remark}{Remark}

\newcommand{\ra}{{\rightarrow}}

\newcommand{\paN}{\mathcal{N}}
\newcommand{\paQ}{\mathcal{Q}}
\newcommand{\pN}{\mathcal{R}}
\newcommand{\paM}{\mathcal{M}}

\newcommand{\paP}{\mathcal{P}}

\newcommand{\email}{Email: } 
\begin{document}

\maketitle

\begin{abstract} \noindent On page 43 in \cite{Po83} Sorin Popa asked whether the following property holds: \emph{If $\omega$ is a free ultrafilter on $\mathbb N$ and $\mathcal \pN_1\subseteq \pN$ is an irreducible inclusion of hyperfinite II$_1$ factors such that $\pN'\cap \pN^\omega\subseteq \pN^\omega_1$ does it follows that $\pN_1=\pN$?} In this short note we provide an affirmative answer to this question. 
\end{abstract}

\section{Introduction}Central sequences were introduced in \cite{MvN36} as a tool to distinguish the hyperfinite II$_1$ factor $\mathcal R$ and the free group factor $L(\mathbb{F}_2)$. Later on, in the groundbreaking papers \cite{MD69a,MD69b, MD69c} D. McDuff analyzed the ultrapower and central sequence algebras of II$_1$ factors to exhibit uncountably many non-isomorphic II$_1$ factors. In his celebrated work \cite{Co76}, A. Connes furthered the study of central sequence algebras and ultrapowers in his proof of \textit{injective implies hyperfinite}, thereby underlining once again the importance of these objects. Since then, the study of ultrapowers and central sequences has played a central role in the theory of II$_1$ factors.
 \vskip 0.04in
 \noindent In $1967$, at the Baton Rouge conference, R.V. Kadison asked a series of influential (yet unpublished!) questions. One of the questions asked wether all maximal amenable subalgebras of a II$_1$ factor are isomorphic to $\mathcal R$. In a seminal paper, \cite{Po83}, S. Popa obtained the striking result that the generator masa in $L(\mathbb{F}_2)$ is maximal amenable, thus answering negatively Kadison's question. In \cite[Theorem 4.1]{Po83} it was also showed that whenever $\mathbb{F}_n \curvearrowright X,$ is a free, measure preserving action on a non-atomic probability space $X$, the $R_u= L^{\infty}(X) \rtimes \langle u\rangle$ is maximal injective subalgebra of $\paM= L^{\infty}(X) \rtimes \mathbb F_{n}$ (where $u$ is a canonical generator of $\mathbb F_n$). The proof relied on showing that if $\paN \supseteq R_u$ is injective subalgebra satisfying $\paN' \cap \paN^{\omega} \subseteq R_u^{\omega}$ then $R_u=\paN$. In turn this  was shown using heavily the notion of \textit{asymptotic orthogonality property} introduced in the same paper. This naturally led S. Popa to ask whether this phenomenon actually occurs in general:
 \textit {Let $\mathcal {R}_1 \subseteq \mathcal R$ be a hyperfinite subfactor such that $\mathcal {R}_1 ' \cap \mathcal R = \mathbb C$ and $\mathcal R' \cap \mathcal R^{\omega} \subseteq \mathcal {R}_1 ^{\omega}$ for some free ultrafilter $\omega$ on $\mathbb N$. Does it follow that $\mathcal R_1 =\mathcal R$}?
See \cite[Section $4.5$ Problem $2$]{Po83}.
 \vskip 0.04in
\noindent In this paper, we answer the aforementioned question in the affirmative (see Theorem \ref{solthm}). Thus the central sequence algebra of the hyperfinite II$_1$ factor cannot be absorbed by some nontrivial irreducible subfactor. Our approach relies upon an interplay between Popa's \textit{deformation/rigidity theory}, subfactor theory, and some basic analysis of central sequences (e.g.\ Ocneanu's central freedom lemma). We believe that this general result may have future applications to maximal amenability questions.

\subsection*{Acknowledgments} The authors warmly thank Adrian Ioana and Jesse Peterson for many helpful discussion, suggestions and for their encouragement. The authors would also like to thank the anonymous referee for his many useful suggestions which greatly improved the exposition of the paper. I.C. was partly supported by NSF Grant DMS \#1600688.

\section{Proof of the Main Result}

\noindent {\large \bf Popa intertwining techniques}. To study the structural theory of von Neumann algberas S. Popa has introduced the following notion of intertwining subalgebras which has been very instrumental in the recent developments in the classification of von Neumann algebras \cite{Po06,Va10icm,Io17icm}. Given (not necessarily unital) inclusions $\paP, \paQ\subseteq \paM$ of von Neumann subalgebras, one says that \emph{a corner of $\paP$ embeds into $\paQ$ inside $\paM$} and writes $\paP\prec_\paM \paQ$ if there exist nonzero projections $ p\in  \paP, q\in \paQ$, a $\ast$-homomorphism $\theta:p \paP p\rightarrow q\paQ q$  and a nonzero partial isometry $v\in q\paM p$ so that $\theta(x)v=vx$, for all $x\in p\paP p$. The partial isometry $v$ is also called an \emph{intertwiner} between $\paP$ and $\paQ$. If we moreover have that $\paP p'\prec_{\paM}\paQ$, for any nonzero projection  $p'\in \paP'\cap 1_\paP \paM 1_\paP$ (equivalently, for any nonzero projection $p'\in\mathcal Z(\paP'\cap 1_\paP \paM 1_\paP)$), then we write $\paP\prec_{\paM}^{s}\paQ$.
\vskip 0.04in

\noindent Then in \cite [Theorem 2.1 and Corollary 2.3]{Po03} Popa developed a powerful analytic method to identify intertwiners between arbitrary subalgebras of tracial von Neumann algebras.

\begin{thm}\cite{Po03} \label{corner} Let $(\paM,\tau)$ be a separable tracial von Neumann algebra and let $\paP, \paQ\subseteq \paM$ be (not necessarily unital) von Neumann subalgebras. 
Then the following are equivalent:
\begin{enumerate}
\item [i)]$\paP\prec_\paM \paQ$.
\item [ii)]For any group $\mathcal G\subset  U(\paP)$ such that $\mathcal G''= \paP$ there is no sequence $(u_n)_n\subset \mathcal G$ satisfying $\|E_{ \paQ}(xu_ny)\|_2\rightarrow 0$, for all $x,y\in  \paM$.
\end{enumerate}
\end{thm} 
\vskip 0.02in

\noindent In order to show our main result we need the following technical result on intertwining. \begin{lem}\label{nointertwining} Let $\omega$ is a free ultrafilter on $\mathbb N$. Let $\paN\subseteq \paM$ be an inclusion of hyperfinite II$_1$ factors such that $\paM\nprec_\paM \paN$. Then we have $\paM'\cap \paM^\omega \nprec_{\paM^\omega} \paN^\omega$. In particular if $\paM'\cap \paM^\omega \subseteq  \paN^\omega$  then $\paM\prec_\paM \paN$.    
\end{lem}
\begin{proof} Since $\mathcal M$ is hyperfinite there exists an ascending sequence of algebras $\paM_n\subseteq \paM$ satisfying $\paM_n\cong M_{2^n}(\mathbb C)$, $\overline{\cup_n\paM_n}^{sot}=\paM$, and $\paM=\paM_n\bar \otimes (\paM_n'\cap \paM)$ for all $n$. Next we briefly argue that $\paM_n'\cap\paM \nprec_\paM \paN$, for all $n $. Assuming otherwise, by \cite[Lemma 2.5]{Va10} there exists a non-zero projection $e\in (\paM_n'\cap \paM)'\cap \paM=\paM_n$ such that $(\paM_n'\cap\paM) e\prec^s_\paM \paN$. Also since $\paM_n$ is finite dimensional then $[e\mathcal Me:(\paM_n'\cap \paM)e]<\infty$ and hence $\paM \prec_\paM (\paM_n'\cap \paM)e$. Using \cite[Remark 3.8]{Va10} we would get that $\paM\prec_\paM\paN$, a contradiction. 
\vskip 0.04in 
\noindent Fix $(s_n)_n\subseteq \mathbb N$ a sequence that tends to $\infty$. Next we claim that for every finite set $F\subset \paM^\omega$ there exists a unitary $v^\omega\in \prod_{n\ra \omega} (\paM_{s_n}'\cap\paM)$ such that $E_{\paN^\omega}(x^\omega v^\omega y^\omega)=0$, for all $x^\omega,y^\omega\in F$.  This relies on the usage of the analytic criterion from Popa's intertwining techniques, i.e.\ part ii) of Theorem \ref{corner}.  Since for every $n\in \mathbb N$ we have $\paM_{s_n}'\cap\paM \nprec \paN$ there exists a unitary $v_n\in  \paM_{s_n}'\cap\paM$ such that $\|E_\paN(x_nv_n y_n)\|_2\leq n^{\text{-1}}$, for all $x^\omega=(x_n)_n,y^\omega=(y_n)_n\in F$. Letting $v^\omega=(v_n)_n\in \prod_{n\ra\omega}\paM_{s_n}'\cap\paM \subset \paM'\cap\paM^\omega$ the previous inequalities show that $E_{\paN^\omega}(x^\omega v^\omega y^\omega)=0$ for all $x^\omega,y^\omega\in F$, as desired. 
\vskip 0.04in 
\noindent Assume by contradiction $\paM'\cap \paM^\omega \prec_{\paM^\omega} \paN^\omega$. Thus one can find projections $0\neq p^\omega \in \paM'\cap \paM^\omega$,$0\neq q^\omega\in  \paN^\omega$, a partial isometry $0\neq w^\omega\in \paM^\omega$, and  a unital $\ast$-homomorphism $\phi: p^\omega (\paM'\cap \paM^\omega) p^\omega \ra  q^\omega\paN^\omega q^\omega$  such that    \begin{equation}\label{601'}\phi(x)w^\omega=w^\omega x\text{ for all }x\in  p^\omega(\paM'\cap \paM^\omega) p^\omega.\end{equation}

\noindent Since $p^\omega\in \paM^\omega=\overline{\cup_n\paM_n}^{sot}$
there exists a sequence $(t_n)_n\subseteq \mathbb N$ that tends to $\infty$ for which $p^\omega \in \prod_{n\ra\omega} \paM_{t_n}$. Using our claim for the sequence $t_n$ and the set $F=\{w^\omega,(w^\omega)^*\}$ one can find a unitary $u^\omega\in \prod_{n\ra \omega} (\paM_{t_n}'\cap\paM)\subseteq \paM'\cap \paM^\omega $ such that $E_{\paN^\omega}(w^\omega u^\omega (w^\omega)^*)=0$. Using this in combination with (\ref{601'}) and $p^\omega u^\omega=u^\omega p^\omega$ we further get that $0=\|E_{\paN^\omega}(w^\omega p^\omega u^\omega p^\omega(w^\omega)^* )\|_2=\|\phi( p^\omega u^\omega p^\omega)  E_{\paN^\omega}(w^\omega (w^\omega)^* )\|_2=\|\phi( u^\omega p^\omega)  E_{\paN^\omega}(w^\omega (w^\omega)^* )\|_2=\|  E_{\paN^\omega}(w^\omega (w^\omega)^* )\|_2 $. This implies that $E_{\paN^\omega}(w^\omega (w^\omega)^* )=0$ and hence $w^\omega =0$, which is a contradiction.  \end{proof}
\begin{remark} \noindent Theorem \ref{corner} also holds without separability assumptions if one uses nets instead of sequences. So the second part of the proof of Lemma \ref{nointertwining} can be directly deduced from Theorem \ref{corner} applied in $\mathcal M^{\omega}$. The authors would like to thank the anonymous referee for pointing this out.
 \end{remark}

\begin{prop}\label{finindexcase}Let $\paN\subseteq \paM$ be II$_1$ factors such that $\paN'\cap \paM =\mathbb C1$. Then $\paM\prec_\paM \paN$ if and only if $[\paM:\paN]<\infty$.\end{prop} 

\begin{proof}  Suppose $\paM\prec_\paM \paN$. Thus one can find nonzero projections $p \in \paM$,$q\in  \paN$, a nonzero partial isometry $v\in q \paM p$, and  a unital $\ast$-homomorphism $\phi:p\paM p\ra q\paN q$ such that \begin{equation}\label{601}\phi(x)v=vx\text{ for all }x\in p \paM p.\end{equation}  Denote by $\mathcal B:= \phi(p\paM p) \subseteq q\paN q$ and notice that by (\ref{601}) we have  $vv^*\in \mathcal B'\cap q\paM q$ and $v^*v\in p\paM p'\cap p\paM p$. Since $\paM$ is a factor we have $v^*v=p$. Moreover by restricting $q$ if necessary we can assume without any loss of generality that the support projection of $E_\paN (vv^*)$ equals $q$.  Equation (\ref{601}) also implies that $\mathcal Bvv^*=vp\paM pv^*= vv^*\paM vv^*$. Since $\paM$ is a factor, this further gives that $ vv^*(\mathcal B'\cap q\paM q )vv^*= (\mathcal Bvv^*)'\cap vv^*\paM vv^*=vv^*( \paM'\cap\paM ) vv^*=\mathbb Cvv^*$. Since $\mathcal B'\cap q\paN q\subseteq \mathcal B'\cap q\paM q$ then there exists  a nonzero projection $r\in \mathcal B'\cap q\paN q$ such that $r (\mathcal B'\cap q\paN q) r= \mathcal B r'\cap r\paN r=\mathbb C r$. Since $q= s(E_\paN(vv^*))$ one can check that $rv\neq 0$. Thus replacing $\mathcal B$ by $\mathcal Br$, $\phi(\cdot)$  by $\phi(\cdot)r$, $q$ by $r$, and $v$ by the partial isometry from the polar decomposition of $rv$ then the intertwining relation (\ref{601}) still holds with the additional assumption that $\mathcal B'\cap q\paN q=\mathbb Cq$. In particular we have that $E_{q\paN q}(vv^*)=c q$ where $c$ is a positive scalar. 

\vskip 0.04in 
\noindent Since $\mathcal B\subseteq q \paN q\subseteq q\paM q$ is an inclusion of II$_1$ factors, $\paN \subseteq \paM$ is irreducible, and  $\mathcal Bvv^*=vv^*\paM vv^*$ then it follows from \cite[Corollary 3.1.9]{Jo81} and \cite[Corollary 1.8]{PP86} that  $\langle q\paN q ,vv^*\rangle\subseteq q\paM q$ is the basic construction of $\mathcal B\subseteq q \paN q$. This entails that $\mathcal B\subseteq q \paN q$ is finite index and moreover $vv^* \langle q\paN q ,vv^*\rangle vv^*=\mathcal Bvv^*= vv^* \paM vv^*$ . Since $\langle q\paN q ,vv^*\rangle$ is a factor then $\langle q\paN q ,vv^*\rangle=q\paM q$  and consequently $q\paN q\subseteq q\paM q$ has finite index. Thus $\paN\subseteq \paM$ has also finite index.  

\vskip 0.04in 
\noindent If $[\paM:\paN]<\infty$ then $\paM\prec_\paM \paN$ follows easily from the fact $\paM_{-1} e_{-1} = e_{-1} \paM e_{-1} $, where $\paM_{-1}$ denotes the downward basic construction, and $e_{-1} \in \paM$ is the corresponding Jones' projection, as in \cite[Corollary 3.1.9]{Jo81}. Note that one does not need $\paN'\cap \paM =\mathbb C1$ for this direction.
\end{proof}

\noindent  For further use we recall next a result due to A. Ocneanu. For a proof the reader may consult \cite[Lemma 15.25]{EK98}.    
\begin{lem}[Ocneanu's central freedom lemma]\label{ocneanu}
		Let $\pN \subseteq \paP \subseteq \paQ$ be separable finite von Neumann algebras, with $\pN$ the hyperfinite factor. If $\omega$ is a free ultrafilter on $\mathbb N$ then we have the following relation 
		\begin{align*}
		(\pN' \cap \paP^{\omega})' \cap \paQ^{\omega}= \pN \vee (\paP' \cap \paQ)^{\omega}.
		\end{align*} 
	\end{lem}

\noindent With these results at hand we are now ready to answer affirmatively Popa's question from \cite{Po83}.

\begin{thm}\label{solthm}
	Let $\paN \subseteq \paM$ be hyperfinite II$_1$ factors such that  $\paN' \cap \paM = \mathbb{C}$. If $\paM' \cap \paM^{\omega} \subseteq \paN^{\omega}$ then $\paN=\paM$.
\end{thm}

\begin{proof} First we notice that from Lemma \ref{nointertwining} and Proposition \ref{finindexcase} it follows that $[\paM:\paN ]<\infty$. Since $\paM' \subseteq \paN'$ and $\paM' \cap \paM^{\omega} \subseteq \paN^{\omega}$ then $\paM' \cap \paM^{\omega}=\paM' \cap \paM^{\omega} \cap \paN' \subseteq \paN^{\omega} \cap \paN'$. So we get the following inclusions:
	\begin{equation}\label{5}
	\paM' \cap \paM^{\omega}  \subseteq \paN' \cap \paN^{\omega} \subseteq \paN' \cap \paM^{\omega}.
	\end{equation}
Since $\paM,\paN$ are McDuff then $\paM' \cap \paM^{\omega} $ and $\paN' \cap \paN^{\omega} $ are von Neumann algebras of type II$_1$. Also, since $\paM' \cap \paM^{\omega} \subseteq \paN' \cap \paM^{\omega}$ then it follows that $\paN' \cap \paM^{\omega}$ is type II$_1$ as well. Also since $\paM$ is hyperfinite then applying Lemma \ref{ocneanu} for $\pN=\paP=\paQ=\paM$ we get $(\paM'\cap \paM^\omega)'\cap\paM^\omega=\paM$ and hence $(\paM' \cap \paM^{\omega})' \cap (\paN' \cap \paM^{\omega}) = \paN' \cap \paM =\mathbb{C}$. In particular, this implies that all algebras displayed in (\ref{5}) are in fact II$_1$ factors. Next we show the following relations 

\begin{equation}\label{6}[\paN' \cap \paM^{\omega}: \paM' \cap \paM^{\omega}] =[\paM:\paN]=[\paN' \cap \paM^{\omega}:\paN' \cap \paN^{\omega}].
\end{equation}

\vskip 0.04in 
\noindent To this end let $\{m_i\}_{1 \leq i \leq n+1}$ be an orthonormal basis of $\paM$ over $\paN$. Then, by \cite[Lemma 3.1]{Po02} it follows that the map $\Phi(x)= [\paM:\paN]^{-1} \sum_i m_ixm_i^*$ implements the conditional expectation from $\paN' \cap \paM^{\omega}$ onto $\paM' \cap \paM^{\omega}$. In addition, the index of $\Phi$ is majorised by $[\paM:\paN]$.  Thus, we get  
\begin{equation}\label{lem:indexineq1}
[\paN' \cap \paM^{\omega}: \paM' \cap \paM^{\omega}] \leq [\paM:\paN].
\end{equation}

	\vskip 0.02in 
\noindent			Now by \cite[Proposition 1.10]{PP86} we have $[\paM:\paN]=[\paM^{\omega}: \paN^{\omega}]$. Set $c= [\paN' \cap \paM^{\omega}:\paN' \cap \paN^{\omega}]^{-1} $ and $\lambda = [\paM^{\omega}: \paN^{\omega}]^{-1}$ and from (\ref{5}) and (\ref{lem:indexineq1}) we infer that $c \geq \lambda$.
				
	\vskip 0.04in 
\noindent	Denote by $E_{\paN'\cap \paM^\omega}$ the conditional expectation from $\paM^{\omega}$ onto $\paN' \cap \paM^{\omega}$ and notice that $E_{\paN'\cap \paM^\omega}\circ E_{\paN^{\omega}}=E_{\paN^{\omega}}\circ E_{\paN'\cap \paM^\omega}$. Let $\paN_n\subseteq \paN$ such that $\paN_n\cong \paM_{2^n}(\mathbb C)$, $\overline{\cup_n \paN_n}^{sot} =\paN$, and $\paN =\paN_n\bar \otimes (\paN_n'\cap \paN)$. Since $\paN_n'\cap \paN\subseteq \paN_n'\cap \paM$ is an inclusion of II$_1$ factors of index $\lambda$ then one can find projections $e_n\in \paN_n'\cap \paM$ such that $E_{\paN_n'\cap \paN}(e_n)=\lambda $ for all $n$. This implies that $E_{ \paN}(e_n)=\lambda $ for all $n$. \vskip 0.04in 
\noindent Altogether, these give $e^\omega=(e_n)_n\in \paN'\cap \paM^\omega$ and $E_{\paN^\omega}(e^\omega)=\lambda$. Thus using \cite[Theorem 2.2]{PP86} we get that $\lambda \geq c$ and hence $\lambda =c$. Summarizing,  
		
	\begin{equation} \label{lem:indexineq2}
		[\paN' \cap \paM^{\omega}:\paN' \cap \paN^{\omega}]= [\paM:\paN].
	\end{equation}

\noindent Altogether, relations (\ref{lem:indexineq1})-(\ref{lem:indexineq2}) conclude relation (\ref{6}).  In turn (\ref{6}) shows that $[\paN' \cap \paN^{\omega}: \paM' \cap \paM^{\omega} ]=1$ and hence 
\begin{equation}\label{7}\paM' \cap \paM^{\omega} =\paN' \cap \paN^{\omega}.\end{equation}
	To finish the proof, we use Lemma \ref{ocneanu}. Indeed setting  $\pN=\paP=\paN$ and $\paQ=\paM$ in Lemma \ref{ocneanu} we get $(\paN' \cap \paN^{\omega})' \cap \paM^{\omega}= \paN$, as $\paN' \cap \paM= \mathbb{C}$. Also letting $\pN=\paP=\paQ=\paM$ in Lemma \ref{ocneanu} we have $(\paM' \cap \paM^{\omega})' \cap \paM^{\omega}= \paM$. Therefore using (\ref{7}) we get $\paN=\paM$, as desired.
	\end{proof}
	


\noindent
\textsc{Department of Mathematics, The University of Iowa, 14 MacLean Hall, Iowa City, IA 52242, U.S.A.}\\
\email {ionut-chifan@uiowa.edu} \\
\email{sayan-das@uiowa.edu}

\end{document}